\documentclass{article}
\usepackage{amsthm}
\usepackage{amssymb}
\usepackage{amsmath}
\usepackage{epsfig}
\usepackage{colonequals}
\usepackage{enumerate}
\usepackage{hyperref}
\usepackage{float}
\usepackage{tikz}
\usepackage{tkz-graph}
\usetikzlibrary{graphs}
\newtheorem{definition}{Definition}
\newtheorem{theorem}{Theorem}

\newtheorem{lemma}[theorem]{Lemma}

\newtheorem{conjecture}[theorem]{Conjecture}
\newtheorem{problem}[theorem]{Problem}
\tikzset{every loop/.style={min distance=20mm, looseness=30}}
\begin{document}

\title{Complete graph immersions and minimum degree}
\author{%
     Zden\v{e}k Dvo\v{r}\'ak\thanks{Computer Science Institute (CSI) of Charles University,
           Malostransk{\'e} n{\'a}m{\v e}st{\'\i} 25, 118 00 Prague, 
           Czech Republic. E-mail: \protect\href{mailto:rakdver@iuuk.mff.cuni.cz}{\protect\nolinkurl{rakdver@iuuk.mff.cuni.cz}}.
           Supported in part by (FP7/2007-2013)/ERC Consolidator grant LBCAD no.
	   616787.}
\and	   
Liana Yepremyan\thanks{School of Computer Science,
McGill University, Montreal, Canada. E-mail:
\protect\href{mailto:liana.yepremyan@mail.mcgill.ca}{\protect\nolinkurl{liana.yepremyan@mail.mcgill.ca}}.}
}
\date{\today}
\maketitle

\begin{abstract}
An \emph{immersion} of a graph $H$ in another graph $G$ is a one-to-one mapping
$\varphi:V(H)\rightarrow V(G)$ and a collection of edge-disjoint paths in $G$,
one for each edge of $H$, such that the path $P_{uv}$ corresponding to the edge
$uv$ has endpoints $\varphi(u)$ and $\varphi(v)$. The immersion is
\emph{strong} if the paths $P_{uv}$ are internally disjoint from
$\varphi(V(H))$. We prove that every simple graph of minimum degree at least
$11t+7$ contains a strong immersion of the complete graph $K_t$. This improves
on previously known bound of minimum degree at least $200t$ obtained by DeVos
et al.\cite{mindim}. Our result supports a conjecture of Lescure and
Meyniel~\cite{lesuremeyniel} (also  independently proposed by Abu-Khzam and
Langston~\cite{abu2003graph}), which is the analogue of famous Hadwiger's
conjecture for immersions and says that every graph without a $K_t$-immersion
is $(t-1)$-colorable.
\end{abstract}

\section{Introduction}

In this paper, \emph{graphs} are simple, without loops and parallel edges,
while \emph{multigraphs} are allowed to have loops and parallel edges.
In both cases, we require that the set of vertices is non-empty and finite.

A graph $H$ is a minor of another graph $G$ if $H$ can be obtained from a subgraph of
$G$ by contracting edges and deleting any resulting loops and parallel edges.
One of the most famous open problems in graph theory, Hadwiger's
conjecture~\cite{hadwiger} from 1943 says that every loopless graph without a
$K_{t}$-minor is $(t-1)$-colorable. This conjecture is widely open for $t\geq 7$;
while $t\le 4$ cases are trivial, $t=5$ case is equivalent to the
celebrated Four-Color Theorem and $t=6$ case was solved by Robertson, Seymour
and Thomas~\cite{robertsonseymourthomas}.   Note that a stronger conjecture by
Haj\'os was proposed in 1940's~\cite{hajos}. A graph $H$ is a \emph{topological
minor} of another graph $G$ if a subgraph of $G$ can be obtained from $H$
by subdividing some edges. Haj\'os conjectured that every graph without a
$K_t$-topological minor must be $(t-1)$-colorable. However, this is known to be
false in general; Catlin~\cite{catlin} disproved this conjecture for all
$t\geq 7$, for $t\leq4$ it follows from the results of Dirac~\cite{dirac}
and $t=5,6$ cases are still open. The topic of this paper is inspired by the
immersion variant of Hadwiger's conjecture proposed by Lescure and
Meyniel~\cite{lesuremeyniel} (and later, independently, by Abu-Khzam and
Langston~\cite{abu2003graph}). 

An \emph{immersion} of a graph $H$ in a graph $G$ is a one-to-one mapping
$\varphi:V(H)\rightarrow V(G)$ and a collection of edge-disjoint paths in $G$,
one for each edge of $H$, such that the path $P_{uv}$ corresponding to the edge
$uv$ has endpoints $\varphi(u)$ and $\varphi(v)$. The vertices in
$\varphi(V(H))$ are called \emph{branch vertices}. We also give an alternative
definition. We define the operation of \emph{splitting off} pairs of edges as
follows. A pair of distinct adjacent edges $uv$ and $vw$ is split off from
their common vertex $v$ by deleting the edges $uv$ and $vw$ and adding the edge
$uw$ (note that this might result in a parallel edge or a loop). We say that
$G$ contains an immersion of the graph $H$ if a graph isomorphic to $H$ can be
obtained from a subgraph of $G$ by splitting off pairs of edges (and removing
isolated vertices). Note that topological minor containment implies both minor and immersion containments, while minor  and immersion containments are
incomparable.
\begin{conjecture}[Lescure and Meyniel~\cite{lesuremeyniel}, Abu-Khzam and Langston~\cite{abu2003graph}]\label{conjectureimmersions}
Every graph without  a $K_t$-immersion is $(t-1)$-colorable.
\end{conjecture}

Conjecture~\ref{conjectureimmersions} is known to be true for $t\le 7$; for
$t\le 4$ the arguments are trivial, for $5\le t\le 7$ it was proven by DeVos
et al.\cite{DKMO} and independently, by Lescure and
Meyniel~\cite{lesuremeyniel} (their proof of $t=7$ case is not published). In
all mentioned cases the authors actually prove stronger statements, showing
that only a lower bound on the minimum degree is required to ensure the existence
of the immersion.

Let $f(t)$ be the smallest value such that every graph of minimum degree
at least $f(t)$ contains an immersion of $K_t$. The idea of considering this
function $f(t)$ was first proposed in~\cite{DKMO}, as the natural analogue of
classical results showing that large average degree (equivalently, large
minimum degree) in a graph implies a $K_t$-minor or $K_t$-topological minor containment. To this
end, it is known that average degree $\Omega(t\sqrt{\log{t}})$ in a graph
forces a $K_t$-minor and this bound is tight
(Kostochka~\cite{kostochka} and Thomason~\cite{thommindeg}). Similarly, average degree $\Omega(t^2)$ forces a topological minor of $K_t$,
as proved independently
by Bollob\'as and Thomason~\cite{BoTh} and by Koml\'os and Szemer\'edi~\cite{KoSz};
again, the bound is tight. 

For immersions, it is easy to see that $f(t)\geq t-1$. In~\cite{DKMO}
and~\cite{lesuremeyniel}, the authors proved that $f(t)=t-1$ when $t=5,6,7$;
this implies Conjecture~\ref{conjectureimmersions}, since every $t$-chromatic graph has a subgraph of minimum degree at least $t-1$.  However,
an example due to Paul Seymour showed that this is not true in general; the
graph obtained from the complete graph $K_{12}$ by removing edges of four
disjoint triangles does not contain an immersion of $K_{10}$. DeVos et al.~\cite{mindim}
generalized this construction, giving graphs of minimum degree
$t-1$ and no $K_t$-immersion for $t=10$ and $t\geq 12$. 
Collins and Heenehan~\cite{collinsheenehan} found infinite families of such
examples for all $t\ge 8$.

Nevertheless, we are not aware of any construction showing a larger gap between $f(t)$ and $t$.
Let us remark that if $f(t)\le t$ for some integer $t$, then
Conjecture~\ref{conjectureimmersions} holds for $t$, see~\cite{abu2003graph}.
Hence, we would like to to pose the following problem.
\begin{problem}
Is it true that for every positive integer $t$, all graphs of minimum degree at least $t$ contain $K_t$ as an immersion?
\end{problem}

DeVos et al.~\cite{mindim} gave the first non-trivial upper bound on $f(t)$, showing that
$f(t)\le 200t$.  Actually, they proved this even for \emph{strong immersions},
when paths $P_{uv}$ are internally disjoint from $\varphi(V(H))$, i.e., the paths $P_{uv}$ are not allowed to
use the branch vertices except as endpoints. They also showed that even
stronger statement is true for dense graphs; every graph with $\Omega(n^2)$ edges
contains an immersion of a clique of linear size such that every path of the immersion has exactly one internal
vertex. In this paper we improve on their bound on $f(t)$ as follows.

\begin{theorem}\label{maintheorem}
For every positive integer $t$, if $G$ is a graph with minimum degree at least $11t+7$, then it contains $K_t$ as a strong immersion.
\end{theorem} 
The rest of the paper is occupied by the proof of this result. In
Section~\ref{sec:prelim} we present all the preliminary results while
Section~\ref{sec:mainproof} contains the main proof of
Theorem~\ref{maintheorem}.

\section{Preliminary results}
\label{sec:prelim}
In this section we present all the auxiliary results necessary for our proof of
Theorem~\ref{maintheorem}.  

DeVos et al.~\cite{mindim} observed that the complete bipartite graph $K_{t,t}$ contains $K_t$ as a strong immersion
(in fact, a slightly more involved argument shows that $K_{t-1,t-1}$ contains $K_t$ as a strong immersion, which
is the best possible, since a graph containing an immersion of $K_t$ must have at least $t$ vertices of degree at
least $t-1$).  We will use two generalizations of this claim.  Before we state them, let us recall a result on list
edge coloring.

\begin{theorem}[H{\"a}ggkvist and Janssen~\cite{haglist}]\label{thm-list}
For every $n\ge 1$, the line graph of $K_n$ has list chromatic number at most $n$.
\end{theorem}

\begin{lemma}\label{lemma-verydense}
Let $t$ be a positive integer.  Let $A$ and $B$ be disjoint sets of vertices of $G$, with $|A|\ge t$.  If each two non-adjacent vertices in $A$ have at least $t$ common neighbors in $B$,
then $G$ contains $K_t$ as a strong immersion, with all branch vertices contained in $A$.
\end{lemma}
\begin{proof}
Let $A_0=\{v_1, \ldots, v_t\}$ be a subset of $A$ of size $t$.  Let $H$ be the complement of $G[A_0]$.
By Theorem~\ref{thm-list}, there exists a proper edge coloring $\varphi:E(H)\to B$ such that $\varphi(uv)$ is a common neighbor of $u$ and $v$ for each $uv\in E(H)$.
For each $uv\in E(H)$, split off in $G$ the pair $u\varphi(uv), v\varphi(uv)$ of edges, obtaining the edge $uv$.  This results in a graph strongly immersed in $G$ such that $A_0$ induces a clique.
\end{proof}

\begin{lemma}\label{lemma-compmult}
Let $t$ be a positive integer.  Every complete multipartite graph $G$ of minimum degree at least $t$ contains $K_t$ as a strong immersion.
\end{lemma}
\begin{proof}
We prove the claim by induction on $t$.  For $t=1$, the claim is obviously true.

Let $V_1$, \ldots, $V_k$ be the parts of $G$.  Note that since the minimum degree of $G$ is at least $t$, we have $|V_2|+\ldots+|V_k|\ge t$.
Let $s=|V_1|$.
If $s\ge t$, then $G$ contains $K_{t,t}$ as a subgraph, and thus $G$ contains a strong immersion of $K_t$.
Suppose that $s\le t-1$.  By the induction hypothesis, $G-V_1$ contains $K_{t-s}$ as a strong immersion $\theta_1$.
Since $|V(G-V_1)|\ge t$, there exists a set $B\subset V(G-V_1)$ of size $s$ that does not contain any branch vertex of $\theta_1$.
By Lemma~\ref{lemma-verydense}, the complete bipartite graph between $V_1$ and $B$ contains a strong immersion $\theta_2$ of $K_s$, with
all branch vertices contained in $V_1$.  Then, $\theta_1$ together with $\theta_2$ and edges between $V_1$ and $V(G-V_1)\setminus B$
form a strong immersion of $K_t$ in $G$.
\end{proof}
Let us remark that the lower bound on minimum degree in Lemma~\ref{lemma-compmult} cannot be improved in general, since e.g. the complete $4$-partite graph
with parts of size three is known~\cite{mindim} not to immerse $K_{10}$.

A graph $G$ is \emph{hypomatchable} if $G-v$ has a perfect matching for every $v\in V(G)$.
We will need Edmonds-Gallai theorem on maximum matchings
in the following form (see e.g. Diestel~\cite{Diestel}, Theorem 2.2.3).

\begin{theorem}[Edmonds-Gallai]\label{thm-edmgal}
Every graph $G$ contains a set $T\subseteq V(G)$ and a matching $M$ of size $|T|$
such that each component of $G-T$ is hypomatchable, each edge of $M$ has exactly one
end in $T$, and no two edges of $M$ have end in the same component of $G-T$.
\end{theorem}

Graphs without a strong immersion of $K_t$ whose complement neither has a perfect matching
nor is hypomatchable have a special structure, as shown in the following lemma; a somewhat weaker form of this result appears implicitly in~\cite{mindim}.

\begin{lemma}\label{lemma-match}
Let $t$ be a positive integer.
Let $G$ be a graph with $n$ vertices that does not contain a complete multipartite subgraph with minimum degree at least $t$.
Suppose that the complement $\overline{G}$ of $G$ neither has a perfect matching nor is hypomatchable,
and let $T\subseteq V(G)$ be as in Theorem~\ref{thm-edmgal} applied to $\overline{G}$.
There exists a non-empty set $W\subseteq V(G)\setminus T$ such that $|T|\le |W|\le t-1$ and each vertex of $W$ has degree at
least $n-t$ and is adjacent in $G$ to all vertices of $V(G)\setminus (T\cup W)$.
\end{lemma}
\begin{proof}
Let $C_1$, \ldots, $C_k$ be the components of $\overline{G}-T$.  
Since the complement of $G$ neither has a perfect matching nor is hypomatchable, we have $k\ge 2$ and $k>|T|$.
Note that $G$ contains a complete multipartite subgraph $G'$ with parts $V(C_1)$, \ldots, $V(C_k)$, and by the assumptions, its
minimum degree is less than $t$.  By symmetry, we can assume that vertices of $V(C_k)$ have degree at most $t-1$ in $G'$,
and thus $|V(C_1)|+\ldots+|V(C_{k-1})|\le t-1$.  Let $W=\bigcup_{i=1}^{k-1} V(C_i)$.

Note that $|W|\ge k-1\ge |T|$.  Each vertex of $W$ is adjacent to all vertices of $V(C_k)=V(G)\setminus (T\cup W)$.
Also, it is adjacent to all vertices of all but one component of $\overline{G}-T$ contained in $W$, i.e., it has degree
at least $k-2\ge |T|-1$ in $G[W]$.  Hence, each vertex of $W$ has degree at least $n-|T|-|W|+(|T|-1)=n-|W|-1\ge n-t$.
\end{proof}

As a first step towards proving Theorem~\ref{maintheorem}, we aim to reduce the problem to Eulerian graphs
(i.e., graphs with only even degree vertices); this is convenient, since even degree vertices can be completely
split off in the process of finding an immersion.  DeVos et al.~\cite{mindim} showed that a graph of minimum degree at least $2d$
contains an Eulerian subgraph of minimum degree at least $d$.  To avoid losing half of the degree in this preprocessing step,
we use a somewhat more involved construction (Lemma~\ref{lemma-eul}) that allows us to only decrease the minimum degree by $6$.

To eliminate vertices of odd degree, we apply the following well-known fact.

\begin{lemma}\label{lemma-parity}
Let $T$ be a tree, and let $f:V(T)\to\{0,1\}$ be arbitrary function such that $\sum_{v\in V(T)} f(v)$ is even.
Then there exists a forest $T'\subseteq T$ such that every vertex $v$ satisfies $\deg_{T'} v\equiv f(v)\pmod{2}$.
\end{lemma}

Hence, to achieve our goal it would suffice to find a spanning tree with bounded maximum degree.  A sufficient condition
for the existence of such a spanning tree was found by Win~\cite{win1989connection}.
Let $c(G)$ denote the number of components of a graph $G$. 

\begin{theorem}[Win~\cite{win1989connection}]\label{thm-win}
Let $k\ge 2$ be an integer.  If every $S\subseteq V(G)$ satisfies $c(G-S)\le (k-2)|S|+2$,
then $G$ has a spanning tree of maximum degree at most $k$.
\end{theorem}

To apply this result, we need to deal with graphs such that removal of a small number of vertices creates many components;
such graphs either have small connectivity or contain vertices of large degree, and the two following lemmas address these
possibilities.  In a graph $G$, for any $X\subseteq V(G)$ we denote
by $\partial{X}$ the set of edges having exactly one endpoint in $X$.  

\begin{lemma}\label{lemma-connsg}
For every even positive integer $d$ and a graph $G$ of minimum degree at least $d$, there exists $X\subseteq V(G)$
such that $|\partial X|<d$ and $G[X]$ is $(d/2)$-edge-connected.
\end{lemma}
\begin{proof}
The claim is trivial if $G$ is $d$-edge-connected.  Otherwise, let $X$ be a smallest non-empty set of vertices
of $G$ such that $|\partial X|<d$.  If $G[X]$ were not $(d/2)$-edge-connected, there would exist a partition
of $X$ to non-empty subsets $X_1$ and $X_2$ such that there are less than $d/2$ edges with one end in $X_1$ and
the other end in $X_2$.  However, then $|\partial X_1|+|\partial X_2|<|\partial X|+2\cdot\frac{d}{2}<2d$, and thus
either $|\partial X_1|<d$ or $|\partial X_2|<d$, contradicting the minimality of $X$.
\end{proof}

\begin{lemma}\label{lemma-nearreg}
Let $t$ and $d$ be positive integers.  If $G$ is a graph of minimum degree at least $d$ and $G$ does not contain $K_t$
as a strong immersion, then $G$ contains as a strong immersion a graph of minimum degree at least $d-1$ and maximum degree at most $d+t$.
\end{lemma}
\begin{proof}
Without loss of generality, we can assume that for every $uv\in E(G)$, either $\deg u=d$ or $\deg v=d$, as otherwise we can remove the edge $uv$.
Let $S$ denote the set of vertices of $G$ of degree at least $d+t+1$, and note that $S$ is an independent set in $G$, by our previous observation.

Consider any $S'\subseteq S$, and let $Y$ be the set of vertices of $G$ adjacent to a vertex in $S'$.  We claim that $|Y|\ge |S'|$:
indeed, the number of edges with one end in $S'$ and the other end in $Y$ is at least $(d+t+1)|S'|$, and at most $d|Y|$.
Hence, Hall's theorem implies that there exists an injective function $g:S\to V(G)$ such that $vg(v)\in E(G)$ for all $v\in S$.

Consider each vertex $v\in S$ in turn, and let $H$ be the subgraph induced by its neighbors.
Each vertex of $H$ has degree at most $d<|V(H)|-t$.  By Lemma~\ref{lemma-match}, we conclude that the complement $\overline{H}$ of $H$ either has perfect matching or is hypomatchable.
In the former case, let $M$ be the perfect matching in $\overline{H}$; in the latter case, let $M$ be the perfect matching in $\overline{H}-g(v)$.
We remove $v$ and add $M$ to the edge set of $G$.  Note that the resulting graph is strongly immersed in $G$.

This way, we eliminated all vertices of degree greater than $d+t$, while we only decreased degrees of vertices of $g(S)$ by one.  It follows that the resulting graph
has minimum degree at least $d-1$ and maximum degree at most $d+t$.
\end{proof}

\begin{lemma}\label{lemma-eul}
Let $t$ be a positive integer and let $d\ge 2t+12$ be an even integer.  If $G$ is a graph of minimum degree at least $d+6$ that does not contain $K_t$ as a strong immersion, then
$G$ contains as a strong immersion an Eulerian graph $G'$ such that $\sum_{v\in V(G')} \max (0, d-\deg v)<d$.
\end{lemma}
\begin{proof}
By Lemma~\ref{lemma-nearreg}, $G$ contains as a strong immersion a graph $G_1$ of minimum degree at least $d+5$ and maximum degree at most $d+t+6$.
By Lemma~\ref{lemma-connsg}, there exists $X\subseteq V(G_1)$ such that $|\partial X|<d$ and $G_1[X]$ is $(d/2)$-edge-connected.
Let $G_2=G_1[X]$.

We claim that $G_2$ has a spanning tree of maximum degree at most $5$.  Indeed, it suffices to verify the assumptions of Theorem~\ref{thm-win}.
Consider any non-empty $S\subseteq X$.  The number of edges with exactly one end in $S$ is at most $|S|(d+t+6)$.
On the other hand, $G_2$ is $(d/2)$-edge-connected, and thus the number of such edges is at least $c(G_2-S)d/2$.  It follows that $c(G_2-S)\le \frac{2(d+t+6)}{d}|S|\le 3|S|$.

Let $T$ be a spanning tree of $G_2$ of maximum degree at most $5$.  By Lemma~\ref{lemma-parity}, there exists $T'\subseteq T$ such that the parity of the degree of each vertex
is the same in $T'$ and in $G_2$.  Hence, $G'=G_2-E(T')$ is an Eulerian graph.  Since $G_1$ has minimum degree at least $d+5$, each vertex of $G'$ except for those incident
with the edges of $\partial X$ has degree at least $d$, and $\sum_{v\in V(G')} \max (0, d-\deg v)\le |\partial X|<d$.
\end{proof}

\section{The Proof of Theorem~\ref{maintheorem}}
\label{sec:mainproof}

Let us first give a brief outline of the proof.  We try to split off the vertices of the graph $G$ one by one, preserving
the degrees of all other vertices and not creating any parallel edges or loops.  This is not possible if the complement of the subgraph
induced by the neighborhood of the considered vertex $a$ does not have a perfect matching; in this case, Lemma~\ref{lemma-match} implies
that $G$ contains vertices that have many common neighbors with $a$.  Let $A$ denote the set of such vertices.  Next, we try to split
off completely all the vertices in $A$, and again, we only fail when some further vertices have many neighbors in common with the vertices of $A$.
Thus, we include these vertices in $A$, and repeat the process.  Eventually, we include at least $t$ vertices in $A$, at which point Lemma~\ref{lemma-verydense}
applies and gives a strong immersion of $K_t$.

This overall structure of the proof is inspired by the proof of DeVos et al.~\cite{mindim}.  The main difference is that in their approach, they allow
splitting only a part of the vertices of $A$, and thus the size of $A$ may increase and decrease throughout the argument.  The termination is ensured
by the fact that splitting the vertices of $A$ increases the density of the graph induced by the common neighbors of $A$, which eventually makes
it possible to find a strong immersion of $K_t$ using another argument specific to very dense graphs.  Avoiding this step enables us to significantly lower
the multiplicative constants (at the expense of a somewhat more complicated analysis of the numbers of common neighbors of vertices of $A$).

\begin{definition}\label{def-state}
Let $t\ge 1$ and $d\ge 11t$ be integers.
A \emph{$(t,d)$-state} is a triple $T=(G,A,B)$, where $G$ is an Eulerian graph such that $\sum_{v\in V(G)} \max (0, d-\deg v)<d$, and
$A\neq\emptyset$ and $B$ are disjoint subsets of vertices of $G$, satisfying the following conditions:
\begin{itemize}
\item[\textrm{(i)}] $d-|A|\le |B|\le d$, and
\item[\textrm{(ii)}] there exists an ordering $a_1$, \ldots, $a_p$ of the elements of $A$ such that
for $1\le i\le p$, the vertex $a_i$ is adjacent to all but at most $|A|+2i$ vertices of $B$.
\end{itemize}
\end{definition}

We say that a set $A\subseteq V(G)$ is \emph{splittable} if there exists a graph $G'$ with vertex set $V(G)\setminus A$
that is strongly immersed in $G$, such that $\deg_{G'} v=\deg_G v$ for every $v\in V(G')$.
A \emph{near-matching} is a graph of maximum degree two, and its \emph{center} is the set of its vertices of degree two.
Let us now formulate the main part of our argument, making precise the claims from the first paragraph of this section.

\begin{lemma}\label{lemma-nosplit}
Let $t\ge 1$ and $d\ge 11t$ be integers.
Let $T=(G,A,B)$ be a $(t,d)$-state, where $G$ does not contain $K_t$ as a strong immersion.  If $A$ is not splittable and $|A|\le t-1$, then there exists a $(t,d)$-state $T=(G,A',B')$
such that $A\subsetneq A'$ and $|A'|<|A|+t$.
\end{lemma}
\begin{proof}
Let $a_1$, \ldots, $a_p$ be the ordering of $A$ from Definition~\ref{def-state}.  We try to completely split off $a_p$, $a_{p-1}$, \ldots, $a_1$ in this order,
obtaining a sequence $G_0=G$, $G_1$, \ldots, $G_m$ of multigraphs strongly immersed in $G$.  We maintain the following invariants for all $i\ge 0$:
\begin{enumerate}
\item $V(G_i)=V(G)\setminus \{a_p,\ldots, a_{p-i+1}\}$ and $E(G_i[B])\setminus E(G)$ is a union of $i$ near-matchings with pairwise disjoint centers; let $Q_i$ denote
the union of these centers.
\item $\deg_{G_i} v=\deg_G v$ for all $v\in V(G_i)$.
\item All loops are incident with vertices of $A$.  Each parallel edge of $G_i$ either has both ends in $A$ or one end in $A$ and multiplicity $2$.
Let $R_i$ denote the set of vertices of $G_i$ not belonging to $A$ that are incident with such a double edge.  Each vertex of $R_i$ is incident with only one double edge,
$Q_i\cap R_i=\emptyset$ and $|Q_i|+|R_i|\le i$.
\end{enumerate}
Clearly, these invariants hold for $G_0$.  Suppose we already constructed $G_{i-1}$.  Let $a=a_{p-i+1}$.  For each double edge joining $a$ with a vertex of $R_{i-1}$,
call one of the edges of the pair \emph{primary} and the other one \emph{secondary}.
Let us define an auxiliary graph $H$ as follows.  The vertices of $H$ are the edges of $G_{i-1}$ incident with $a$ such that their other ends belong to $V(G_{i-1}-A)$.
Two distinct edges $e_1=au$ and $e_2=av$ are adjacent in $H$ if either $uv\in E(G_{i-1})$ or $u,v\in R_{i-1}$ and at least one of $e_1$ and $e_2$ is secondary (including the case $u=v$, see Figure~\ref{fig:primsec} for an illustration).
\begin{figure}[H]
\centering
\begin{tikzpicture}
[scale=.5,auto=left,every node/.style={circle,fill=,scale=.6}]
\node[label=above:\large $u$](nu) at (1,10) {};
\node[label=above:\large $v$] (nv) at (9,10) {};
\node[label=below:\large $a$] (na) at (5,6)  {};
\path (na) edge [bend left,color=blue,line width=1pt] (nv);
\path (nv) edge [bend left,color=red,line width=1pt] (na);
\path (na) edge [bend left,color=blue,line width=1pt](nu);
\path (nu) edge [bend left,color=red,line width=1pt](na);
\draw (nu) [dashed] edge (nv);
\end{tikzpicture}
\hspace{2cm}\begin{tikzpicture}
[scale=.5,auto=left,every node/.style={circle,fill=,scale=.6}]
\node[label=above:\large $au$,fill=blue](bau) at (1,10) {};
\node[label=above:\large $av$,fill=blue] (bav) at (9,10) {};
\node[label=below:\large $au$,fill=red](rau) at (1,6) {};
\node[label=below:\large $av$,fill=red] (rav) at (9,6) {};
\draw (rav) edge (rau);
\draw (rav) edge (bau);
\draw (rau) edge (bav);
\draw (rav) edge (bav);
\draw (rau) edge (bau);
\draw (bau) [dashed] edge (bav);
\end{tikzpicture}
\caption{Primary (blue) and secondary(red) edges in $G_{i-1}$ and the corresponding subgraph in $H$(assuming $uv\notin E(G_{i-1})$)}
\label{fig:primsec}
\end{figure}
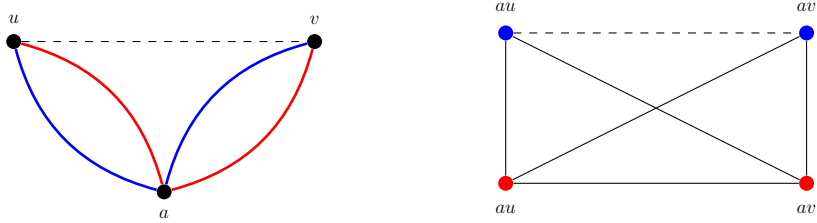
Let us first consider the case when the complement
$\overline{H}$ of $H$ either has a perfect matching or is hypomatchable.
Note that there exists $e\in V(\overline{H})\setminus (Q_{i-1}\cup R_{i-1})$, since $a$ has at least $|B|-\bigl(|A|+2(p-i+1)\bigr)\ge d-4t>t>i$ distinct neighbors in $B$.
If $\overline{H}$ has a perfect matching, then let $M$ be such a perfect matching.
If $\overline{H}$ is hypomatchable, then let $M$ be a perfect matching in $\overline{H}-e$. 
Let $M'$ be the graph with vertex set $V(G_{i-1}-A)$ and edge set
$\{xy:e_1e_2\in M, e_1=ax, e_2=ay\}$.

Note that the construction of $H$ implies that $M'$ is a near-matching contained
in the complement of $G_{i-1}$ and that the center $Q$ of $M'$ is a subset of $R_{i-1}$.
The graph $G_i$ is obtained from $G_{i-1}$ by splitting off $a$ completely so that the edges from $a$ to $V(G_{i-1}-A)$ (except for $e$ when $\overline{H}$ is hypomatchable)
give rise to $M'$ and the rest of edges (those to $A$ and the edge $e$ if $\overline{H}$ is hypomatchable) are split off arbitrarily (it is possible to split off $a$ completely, since it has even degree
and we do not restrict parallel edges and loops with both ends in $A$).  
The double edges incident with the center of $M'$ were split off and
at most one new parallel edge with an end outside of $A$ is created; this happens in the case when $\overline{H}$ is hypomatchable and
$e=az$ ends up to be split off to form such a parallel edge.  By our original choice of $e$, $z\not\in R_{i-1}$, thus $z$ is incident with at most one parallel edge in $G_i$ and such a parallel edge has multiplicity $2$.
Therefore, letting $Q_i=Q_{i-1}\cup Q$, we have $R_i\subseteq (R_{i-1}\setminus Q)\cup\{z\}$, and thus $Q_i\cap R_i=\emptyset$ and $|Q_i|+|R_i|\le |Q_{i-1}|+|R_{i-1}|+1\le i$.
Hence all the invariants are satisfied.

Since $A$ is not splittable, we cannot split off all vertices of $A$ in this way---the invariants imply that $G_p$ would be a graph with vertex set $V(G)\setminus A$
that is strongly immersed in $G$, such that $\deg_{G_p} v=\deg_G v$ for every $v\in V(G_p)$.  Hence, for some $i\le p$, the graph $\overline{H}$ neither has a perfect matching nor
is hypomatchable.  Note that $H$ does not contain a complete multipartite subgraph with minimum degree at least $t+|R_{i-1}|\le 2t$: otherwise, after removing vertices corresponding to the secondary edges we obtain a complete multipartite subgraph with minimum degree at least $t$ which, by definition, corresponds to a complete multipartite subgraph with minimum degree at least $t$ in $G_{i-1}$, in contrary to Lemma~\ref{lemma-compmult}.

Let $T$ and $W$ be the sets from Lemma~\ref{lemma-match} applied to $H$, where $|T|\le |W|<2t$ and let $W'=\{w\in V(G_{i-1}):aw\in W\}$. Clearly $W'\neq \emptyset$. We will show that $(G,A',B')$ is a $(t,d)$-state for  $A'=A\cup W'$ and $B'=B\setminus W'$.

First we check that $|W'|\leq t-1$. Let $T'=\{t\in V(G_{i-1}):at\in T\}$, $R'=R_{i-1}\cap W'$ and $R''=R_{i-1}\setminus R'$. Note that $W'\neq\emptyset$ and since each edge in $W$ has multiplicity at most $2$,
we have $|T|\le|W|\le |W'|+|R'|$.  Furthermore, each vertex of $W'$ is adjacent in $G_{i-1}$ to all vertices of of $(N(a)\cap B)\setminus (W'\cup T'\cup R'')$. But we have
\begin{align*}
|(N(a)\cap B) \setminus (T'\cup W'\cup R'')| &\ge |N(a)\cap B|-|T|-|W|-|R_{i-1}| \\
&\ge |B|-|B\setminus N(a)|-2|W|-|R_{i-1}| \\
&\ge d-9t> t.
\end{align*}
Hence if $|W'|\ge t$, then $G_{i-1}$ contains $K_{t,t}$ as a subgraph which implies that $G$ contains $K_t$ as a strong immersion, a contradiction.  

So to finish the proof of the lemma it remains to check conditions (i)-(ii) in Definition~\ref{def-state}. The condition (i) holds trivially. For condition (ii), let $A'=\{a_1,\ldots, a_{p+k}\}$, where
$a_{p+1}$, \ldots, $a_{p+k}$ is an arbitrary ordering of $W'$.  For $1\le j\le p$, we trivially have $|B'\setminus N(a_j)|\le |B\setminus N(a_j)|\le |A|+2j\le |A'|+2j$.

Now suppose $p+1\le j\le p+k$.  In $G_{i-1}$, $a_j$ is adjacent to all vertices of $(N(a)\cap B)\setminus (W'\cup T'\cup R'')$.
Thus, $a_j$ has at most $|B\setminus N(a)|+|T'|+|R''|$ non-neighbors in $B'$ in the graph $G_{i-1}$.  Since $E(G_{i-1}[B])\setminus E(G)$ is a union of $(i-1)$ near-matchings with pairwise distinct centers,
$G_{i-1}$ contains at most $i$ edges incident with $a_j$ that do not belong to $G$.  Recall that $|T'|\le |W'|+|R'|$.  Hence in $G$ we have
\begin{align*}
|B'\setminus N(a_j)|&\leq |B\setminus N(a)|+|T'|+|R''|+i\\
&\le |A|+2(p-i+1)+|W'|+|R_{i-1}|+i\\
&\le |A|+2(p-i+1)+|W'|+2i-1 \\
&=|A'|+2p+1\\
&\le |A'|+2j,\end{align*} as desired.
\end{proof}

It remains to show that there exists some $(t,d)$-state to which Lemma~\ref{lemma-nosplit} can be applied, and that
the number of applications of Lemma~\ref{lemma-nosplit} is bounded.  These facts follow from the next two lemmas.

\begin{lemma}\label{lemma-exstart}
Let $d$ be a positive integer. If $G$ is a graph such that $\sum_{v\in V(G)} \max (0, d-\deg v)<d$, then $G$ contains a vertex of
degree at least $d$.
\end{lemma}
\begin{proof}
Let $X$ be the set of vertices of degree less than $d$ in $G$.  The claim is trivial if $X=\emptyset$, hence assume that there exists a vertex
$x\in X$.  If all neighbors of $x$ belonged to $X$, then since the graph $G$ is simple, we would have $\deg x\le \sum_{v\in V(G)\setminus\{x\}} \max (0, d-\deg v)$,
and thus
$$\sum_{v\in V(G)} \max (0, d-\deg v)=d-\deg x + \sum_{v\in V(G)\setminus\{x\}} \max (0, d-\deg v)\ge d,$$
contradicting the assumptions.  Therefore, $x$ has a neighbor $v$ not belonging to $X$, i.e., $\deg v \ge d$.
\end{proof}

\begin{lemma}\label{lemma-exsplit}
Let $t\ge 1$ and $d\ge 11t$ be integers.
Let $T=(G,\{a\},B)$ be a $(t,d)$-state. If $G$ does not contain $K_t$ as a strong immersion, then there exists a splittable set $A\subsetneq V(G)$ with $a\in A$.
\end{lemma}
\begin{proof}
Note that by Lemma~\ref{lemma-exstart}, $G$ has at least $d+1$ vertices.  If $G$ does not contain such a splittable set, then repeated applications of Lemma~\ref{lemma-nosplit}
give us a $(t,d)$-state $(G,A',B')$, where $a\in A'$ and $t\le |A'|\le 2t$.  Let $a_1$, \ldots, $a_p$ be the ordering of $A'$ as in Definition~\ref{def-state}.
For $1\le i<j\le t$, we have $|B'\setminus N(a_i)|\le |A'|+2i\le 4t$ and $|B'\setminus N(a_j)|\le 4t$, and thus $a_i$ and $a_j$ have at least $|B'|-8t\ge d-|A'|-8t\ge t$ common neighbors in $B'$.
However, by Lemma~\ref{lemma-verydense}, this implies that $G$ contains $K_t$ as a strong immersion, which is a contradiction.
\end{proof}

Combining these results, we now easily obtain a strong immersion of $K_t$ as required.

\begin{lemma}\label{lemma-eulex}
Let $t\ge 1$ and $d\ge 11t$ be integers.
If $G$ is an Eulerian graph such that $\sum_{v\in V(G)} \max (0, d-\deg v)<d$, then $G$ contains $K_t$ as a strong immersion.
\end{lemma}
\begin{proof}
Suppose for a contradiction that $G$ does not contain $K_t$ as a strong immersion, and let $G$ be such a graph with the smallest number of vertices.
By Lemma~\ref{lemma-exstart}, there exists $v\in V(G)$ of degree at least $d$.  Let $B$ be a set of $d$ neighbors of $v$.
Then $T=(G,\{v\},B)$ is a $(t,d)$-state.  By Lemma~\ref{lemma-exsplit}, there exists a non-empty splittable $A\subsetneq V(G)$.
Hence, there exists a graph $G'$ with vertex set $V(G)\setminus A$ that is strongly immersed in $G$, such that $\deg_{G'} v=\deg_G v$ for every $v\in V(G')$,
and in particular $G'$ is Eulerian and $\sum_{v\in V(G')} \max (0, d-\deg v)<d$.  However, by the minimality of $G$, the graph $G'$ contains $K_t$ as a strong immersion, which
is a contradiction.
\end{proof}

\begin{proof}[Proof of Theorem~\ref{maintheorem}]
Suppose for a contradiction that $G$ does not contain $K_t$ as a strong immersion.
Let $d\in\{11t,11t+1\}$ be even.  By Lemma~\ref{lemma-eul}, $G$ contains as a strong immersion an Eulerian graph $G'$ such that $\sum_{v\in V(G')} \max (0, d-\deg v)<d$.
However, then $G'$ contains $K_t$ as a strong immersion by Lemma~\ref{lemma-eulex}, which is a contradiction.
\end{proof}

\bibliographystyle{acm}
\bibliography{../data.bib}

\end{document}